\documentclass[11pt,a4paper]{article}
\usepackage[cp1251]{inputenc}
\usepackage{amsmath,amsthm}
\usepackage{amsfonts,amstext,amssymb,verbatim,epsfig}
\theoremstyle{usual}
\newtheorem{theorem}{Theorem}

\newtheoremstyle{likedef}
  {}%
  {}%
  {}%
  {\parindent}%
  {\bfseries}%
  {.}%
  {.5em}%
  {}%

\theoremstyle{likedef}

\numberwithin{equation}{section}

\begin{document}

\title{Upper bound on the expected size of intrinsic ball}

\author{Art\"{e}m Sapozhnikov\thanks{EURANDOM, P.O. Box 513, 5600 MB Eindhoven, The Netherlands. Email: sapozhnikov@eurandom.tue.nl;
Research partially supported by Excellence Fund Grant of TU/e of Remco van der Hofstad.}}
\date{}
\maketitle


\begin{abstract}
We give a short proof of Theorem~1.2(i) from \cite{KN(IIC)}.
We show that the expected size of the intrinsic ball of radius $r$ is at most $Cr$ if
the susceptibility exponent $\gamma$ is at most $1$. In particular, this result follows if the so-called triangle condition holds.\\\\
\end{abstract}

Let $G=(V,E)$ be an infinite connected bounded degree graph.
We consider independent bond percolation on $G$.
For $p\in [0,1]$, each edge of $G$ is open with probability $p$ and closed with probability $1-p$ independently for distinct edges.
The resulting product measure is denoted by ${\mathbb P}_p$.
For two vertices $x,y\in V$ and an integer $n$, we write $x\leftrightarrow y$ if there is an open path from $x$ to $y$, and
we write $x\stackrel{\leq n}\longleftrightarrow y$ if there is an open path of at most $n$ edges from $x$ to $y$.
Let $C(x)$ be the set of all $y\in V$ such that $x\leftrightarrow y$.
For $x\in V$, the {\it intrinsic ball} of radius $n$ at $x$ is the set $B_I(x,n)$ of all $y\in V$ such that $x\stackrel{\leq n}\longleftrightarrow y$.
Let $p_c = \inf \{p~:~{\mathbb P}_p(|C(x)|=\infty)>0\}$ be the critical percolation probability.
Note that $p_c$ does not depend on a particular choice of $x\in V$, since $G$ is a connected graph.
For general background on Bernoulli percolation we refer the reader to \cite{Grimmett}.

In this note we give a short (and slightly more general) proof of Theorem~1.2(i) from \cite{KN(IIC)}.
\begin{theorem}\label{thmBI}
Let $x\in V$.
If there exists a finite constant $C_1$ such that ${\mathbb E}_{p}|C(x)| \leq C_1 (p_c - p)^{-1}$ for all $p<p_c$, then
there exists a finite constant $C_2$ such that for all $n$,
\[
{\mathbb E}_{p_c}|B_I(x,n)| \leq C_2 n.
\]
\end{theorem}
Before we proceed with the proof of this theorem, we discuss examples of graphs for which the assumption of Theorem~\ref{thmBI} is known to hold.
This assumption can be interpreted as the mean-field bound $\gamma \leq 1$, where $\gamma$ is the susceptibility exponent.
It is well known that for vertex-transitive graphs this assumption is satisfied if the triangle condition holds at $p_c$ \cite{AN84}: For $x\in V$,
\[
\sum_{y,z\in V} {\mathbb P}_{p_c}(x\leftrightarrow y) {\mathbb P}_{p_c}(y\leftrightarrow z) {\mathbb P}_{p_c}(z\leftrightarrow x) < \infty.
\]
This condition holds on certain Euclidean lattices \cite{HS,HHS} including the nearest-neighbor lattice $\mathbb Z^d$ with $d\geq 19$ and sufficiently spread-out lattices with $d>6$.
It also holds for a rather general class of non-amenable transitive graphs \cite{K,S1,S2,W}.
It has been shown in \cite{Ktot} that for vertex-transitive graphs, the triangle condition is equivalent to the open triangle condition.
The latter is often used instead of the triangle condition in studying the mean-field criticality.

\begin{proof}[Proof of Theorem~\ref{thmBI}]
Since $G$ is a bounded degree graph, it is sufficient to prove the result for $n\geq 2/p_c$.
Let $p<p_c$. We consider the following coupling of percolation with parameter $p$ and with parameter $p_c$.
First delete edges independently with probability $1-p_c$, then every present edge is deleted independently with probability $1 - (p/p_c)$.
This construction implies that for $x,y\in V$, $p<p_c$, and an integer $n$,
\[
{\mathbb P}_{p}(x\stackrel{\leq n}\longleftrightarrow y) \geq \left(\frac{p}{p_c}\right)^n {\mathbb P}_{p_c}(x\stackrel{\leq n}\longleftrightarrow y).
\]
Summing over $y\in V$ and using the inequality ${\mathbb P}_{p}(x\stackrel{\leq n}\longleftrightarrow y) \leq {\mathbb P}_p(x\leftrightarrow y)$, we obtain
\[
{\mathbb E}_{p_c}|B_I(x,n)| \leq \left(\frac{p_c}{p}\right)^n {\mathbb E}_p|C(x)|.
\]
The result follows by taking $p = p_c - \frac{1}{n}$.
\end{proof}

\textbf{Acknowledgements.} I would like to thank Takashi Kumagai for valuable comments and advice.

\end{document}